\documentclass{amsart}

\usepackage{amssymb,amsmath,amscd}
\usepackage{a4wide}
\usepackage{graphicx}  
\usepackage{epsfig,psfrag}  
\usepackage{bbm}
\usepackage{textcomp}
\usepackage[active]{srcltx}
\usepackage{enumerate}
\usepackage{type1cm}

\newtheorem{theorem}{Theorem}
\theoremstyle{definition}
\newtheorem{defin}{Definition}
\newtheorem{prop}{Proposition}
\newtheorem{remark}{Remark}
\newtheorem{lemma}{Lemma}
\newtheorem{corollary}{Corollary}

\def\v{v^{}_{G,x,y}}
\def\Zt{(Z_t)^{}_{t\ge 0}}
\newcommand{\Id}{\mathbbm 1}
\newcommand{\btimes}{\mbox{\huge \raisebox{-0.2ex}{$\times$}}}
\newcommand{\cM}{\mathcal{M}}
\def\EE{\mathbb{E}}

\begin{document}

\title{Moment closure in a Moran model with recombination}

\author{Ellen Baake}
\address[Ellen Baake]{Technische Fakult\"at, Universit\"at Bielefeld, 
Box 100131, 33501 Bielefeld, Germany}
\email[Ellen Baake]{ebaake@techfak.uni-bielefeld.de}

\author{Thiemo Hustedt}
\address[Thiemo Hustedt]{Forschungsschwerpunkt Mathematisierung, Universit\"at Bielefeld, 
Box 100131, 33501 Bielefeld, Germany}
\email[Thiemo Hustedt]{thustedt@uni-bielefeld.de}



\begin{abstract}
We extend the Moran model with single-crossover recombination to include general recombination and mutation. We show that, in the case without resampling, the expectations of products of marginal processes defined via partitions of sites form a closed hierarchy, which is exhaustively described by a finite system of differential equations. One thus has the exceptional situation of moment closure in a nonlinear system. Surprisingly, this property is lost when resampling (i.e., genetic drift) is included.

\end{abstract}

\maketitle

\section{Introduction}

In recent years, the processes of population genetics, which describe the
genetic structure of populations under the influence of evolutionary
forces such as mutation, selection, recombination, migration, and
genetic drift, have been a rich source of fascinating probabilistic
problems. More precisely, the dynamics is often well understood in the
limit of infinite population size, where a law of large numbers leads to
a deterministic description (in terms of discrete dynamical systems or
differential equations), but great challenges ensue if the population is
finite, in particular if there is interaction between individuals, such as
competition (selection) or recombination (the combination of genetic
material of two parents into the `mixed' genetic type of an offspring); see \cite{burger, durrett, ewens}. Interactions usually make the infinite-population model
nonlinear and, often, already difficult enough to treat. In the
corresponding stochastic model, they are reflected by transition rates
(or probabilities) that depend nonlinearly on the current state of the
system and often result in processes whose treatment provides enormous
challenges. Even the relationship between the stochastic process and its
deterministic counterpart is usually unclear (apart from the infinite
population limit). In particular, the expectation of the stochastic
process is, usually, not given by the corresponding deterministic
dynamics - in general, such coincidence is reserved for populations of
individuals that evolve independently (as in branching processes); or
systems with interactions that do not change the expectation (like
Wright-Fisher sampling).

Indeed, even the analysis of the expectation is difficult in most processes
of population genetics with interaction. Its dynamics does, usually, not
only depend on the current expectation, but on higher moments, whose
change, in turn, depends on even higher moments. Formulating this
hierarchy of dependencies is
a common approach for stochastic processes arising in various
applications in physics, chemistry, and biology \cite{levermore, chang, dieckmann}.
Usually, this hierarchy continues indefinitely (it does not `close'); to
extract at least an approximation to the (lower) moments of interest,
some method of `moment closure' must be employed (in the simplest case, a
truncation) \cite{dieckmann}.

The corresponding deterministic systems (that arise through a law of large numbers) are also often tackled via systems of moments or cumulants, see \cite[Ch. V.4]{burger} for an overview. Models of recombination take a special role between
linear and nonlinear models. Although there is abundant interaction and
hence nonlinearity, the deterministic system that describes the
frequencies of all possible (geno)types may be (exactly) transformed
into a linear one by embedding it into a higher-dimensional space (more
explicitly, by adding further components that correspond to products of
type frequencies). This method is known as Haldane linearisation \cite{HaleRingwood}. The
underlying linear structure even allows a diagonalisation and explicit
solution, see \cite{ute} and references therein. In certain important special
cases (notably, in so-called single-crossover dynamics in continuous time), this solution is surprisingly simple and immediately plausible \cite{reco, MB}.

Elucidating underlying linear structures in the corresponding stochastic
system (more precisely, in the Moran model with recombination) has only
started very recently. In the aforementioned single-crossover case, Bobrowski et al. \cite{bobrowski} analysed the asymptotic behaviour  in the presence of mutation. Baake and Herms \cite{baakeherms} observed that the expected type frequencies in the finite system (but without genetic drift)
follow those in the deterministic model; this could be explained by the
(conditional) independence of certain marginalised processes that appear
as `subsystems' of the stochastic model. This and other results now lead to
the question whether in the general recombination scheme (i.e., not restricted to single crossovers) the dynamics of the expectations may be embedded into a higher but finite dimensional space, such that they are given by a finite system of differential equations? Is there an equivalent of Haldane
linearisation in the sense of moments?

This article will address these questions in the framework of the Moran
model with recombination and mutation. In particular it will show that the system of moments closes here after a finite number of steps, without any need for approximations, as long as there is no genetic drift. This may be considered as a
stochastic analogue of Haldane linearisation.

\section{Moran model with recombination}

We consider a population of $N$ individuals. Each of them is endowed with the set $S=\{1,\dots,n\}$ of sites. These can be interpreted as nucleotide positions in a string of DNA or as gene loci on a chromosome. For each site $i$ there is a finite set $X_i$ of  alleles that may occur at site $i$. A string of alleles is then called a type,  $X:=\btimes^n_{i=1}X_i$ is the  type space.

We are interested in modelling  recombination, which means the rearrangement of genetic material in sexually reproducing populations. It may occur during meiosis, the creation of gametes, that is egg cells or sperm. Homologous chromosomes may cross over at some points and exchange the genetic material in between (see Figure \ref{rekoschema}).

\begin{figure}
\begin{center}
\includegraphics[scale=0.4]{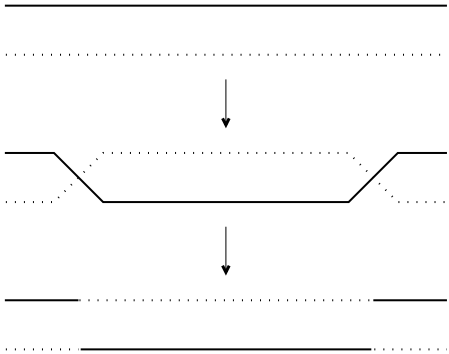}
\caption{Recombination}\label{rekoschema}
\end{center}
\end{figure}

In the following we will assign recombination events to subsets of sites in a natural way. Let $G\subset S$. Then the corresponding recombination event between two individuals is the following: the alleles at the sites given by $G$ remain at their positions, whereas the alleles at the sites in $\bar G$, the complement of $G$, are exchanged (see Figure \ref{sites_reco}). 
\begin{figure}
 \begin{center}
  \includegraphics[scale=0.4]{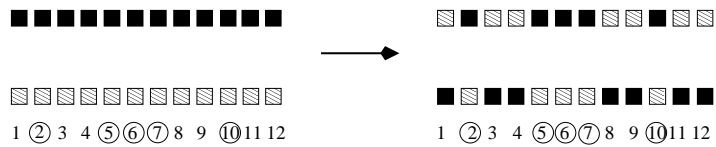}
\caption{Recombination event defined by the set $G$ (circled sites)}\label{sites_reco}
 \end{center}
\end{figure}

We define the mappings $p^{}_G:X\times X\rightarrow X$, $G\subset S$ by
\begin{equation}\label{pG}
p^{}_G(x,y)=\ :\Big(\btimes_{i\in G}\{x^{}_i\}\Big) \times \Big(\btimes_{i\in \bar G}\{y^{}_i\}\Big): \ ,
\end{equation}
where   $:\dots :$ means that the coordinates are ordered as in $X$. So, $p^{}_G(x,y)$ and $p^{}_{\bar G}(x,y)$ are the new types resulting from the recombination event corresponding to $G$ between the types $x$ and $y$.
Obviously, $p^{}_G(x,y)=p^{}_{\bar G}(y,x)$. So, $G$ and $\bar G$ essentially correspond to the same recombination event.

We now define a Moran model with recombination and mutation. Each individual undergoes  recombination events corresponding to $G\subset S$ at rate $\varrho^{}_G/4\ge 0$ for all $G\subset S$. The recombination partner is chosen out of the whole population (including the opening individual itself). Then they exchange their genetic material according to the recombination event corresponding to $G$ (see Figure \ref{moranmodel}). 
\begin{figure}
 \begin{center}
\psfrag{rho}{$\varrho^{}_G$}
\psfrag{N ind}{$N$ individuals}
\psfrag{pg}{{$p^{}_G(x,y)$}}
\psfrag{pgq}{{$p^{}_{\bar G}(x,y)$}}
\psfrag{t}{$t$}
\psfrag{t1}{$t_1$}
\psfrag{punkte}{$\cdot\cdots\cdot$}
  \includegraphics[scale=0.4]{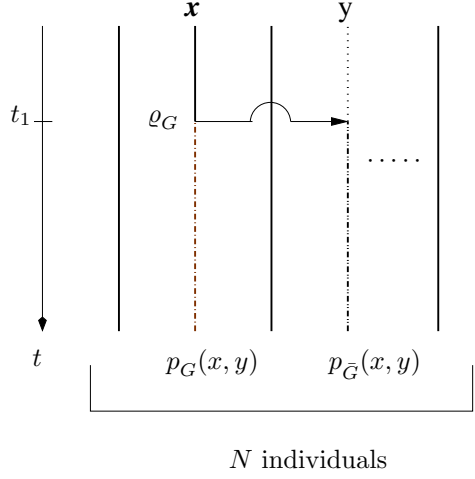}
\caption{Moran model with recombination. At time $t^{}_1$ the second individual, which is of type $x$, undergoes a recombination event corresponding to $G$ and chooses its partner randomly, here the fourth individual, which is of type $y$; from that time on, the individuals are of type $p^{}_G(x,y)$ and $p^{}_{\bar G}(x,y)$.}\label{moranmodel}
 \end{center}
\end{figure}
To keep things well-defined, the recombination rates $\varrho^{}_G$ have the properties $\varrho^{}_G=\varrho^{}_{\bar G}$ and $\varrho^{}_\varnothing=\varrho^{}_S=0$. 

Furthermore, mutation events may occur. An allele $x^{}_i\in X_i$ at site $i$ mutates into allele $y^{}_i\in X_i$ with rate $\mu^{i}_{x^{}_iy^{}_i}\ge 0$. Thus, the mutation rate depends on both the parental and the offspring allele. 

Additionally, we introduce birth events or, more precisely, resampling. Each individual produces an offspring at rate $b/2\ge 0$. The offspring inherits the parent's type and replaces another individual, randomly chosen from the entire population (again including the parent individual).

In the following we are interested in the composition of the population, so we define the stochastic process $\Zt$   with state space $E:=\{\omega \ \text{counting measure on} \ X\ \text{with}\ \omega(X)=N\}$, by
\[Z_t(\{x\}):=  \text{number of individuals of type} \ x.\]
In the following we will use shorthands like $Z_t(x)$, $z(x)$ instead of $Z_t(\{x\})$, $z(\{x\})$. Recombination, mutation and resampling events  induce the following transitions if $Z_t=z$:

\begin{equation}\label{recombination}
\begin{split}
 & z\rightarrow z + \v \quad \text{with}\quad\v:= -\delta_x-\delta_y + \delta^{}_{p^{}_G(x,y)} +\delta^{}_{p^{}_{\bar G} (x,y)}\\
 & \text{at rate} \quad\frac{1}{N} \varrho^{}_G z(x)z(y)\quad \text{for} \quad x,y \in X, \quad G\subset S,
\end{split}
\end{equation}

\begin{equation}\label{mutation}
z\rightarrow z - \delta^{}_{(x^{}_1,\dots,x^{}_i,\dots,x^{}_{n})} + \delta^{}_{(x^{}_1,\dots,y^{}_i,\dots,x^{}_n)}
\quad\text{at rate}\quad \mu^i_{x_iy_i} z(x),
\end{equation}

\begin{equation}\label{resampling}
 z\rightarrow z+\delta^{}_x-\delta^{}_y \quad\text{at rate} \quad \frac{b}{2N}z(x)z(y). 
\end{equation}
The rate in \eqref{recombination} is determined in the following way: An individual of type $x$ recombines at rate $\frac{1}{4}\varrho^{}_Gz(x)$ and chooses one individual of type $y$ with probability $\frac{z(y)}{N}$. This leads to the rate $\frac{1}{4N} \varrho^{}_G z(x)z(y)$ which needs to be multiplied by $4$ to account for the fact that the recombination could be initiated by an individual of type $y$ and that recombination according to $G$ is the same as recombination according to $\bar G$.

A brief comment on the model is in order. We consider recombination and reproduction as independent events whereas, in true biology, recombination is coupled to reproduction. We use the decoupled version here because it is simpler, and because it allows to clearly separate the effects of random recombination from those of random reproduction. This version is also used elsewhere \cite{PHW06}, on the argument that recombination events are rare.

For a subset $G\subset S$ we define $X_G:=\btimes_{i\in G}X_i$ and  the mapping $\pi^{}_G:X\rightarrow X_G$ as the canonical projection.
Let $\omega$ be a (signed) measure on $X$. We define the pullback $\pi^{}_G.$ by $\pi^{}_G.\omega:=\omega\circ\pi^{-1}_G$. So, $\pi^{}_G.$ maps a measure on $X$ onto its corresponding marginal measure on $X_G$.

In the following, marginal processes of $Z_t$ will play a crucial role. The following proposition states that these are Markov chains, too. It is an extension of Lemma 1 in \cite{baakeherms}.
\begin{prop}
Let $I\subset S$, and let $\Zt$ be the recombination process as defined by equations \eqref{recombination}-\eqref{resampling}. Then $(\pi^{}_I.Z_t)^{}_{t\ge 0}$ is a Markov process with state space $E_I:=\{\omega\ \text{ \rm counting measure on }X_I$, $\omega(X_I)=N\}$.
\end{prop}
\begin{proof}
Obviously,  $(\pi^{}_I.Z_t)_{t\ge0}$ is a stochastic process on $E_I$.

We must show that the transition rates of $(\pi^{}_I.Z_t)_{t\ge 0}$ only depend on the current state of the process. A recombination event induces the following transition:
\[\pi^{}_I.z\rightarrow \pi^{}_I.(z+\v),\]
with
\begin{equation}\label{induziert}
 \pi^{}_I.\v=\delta_{\pi^{}_I(p^{}_G(x,y))}+ \delta_{\pi^{}_I(p^{}_{\bar G}(x,y))} -\delta_{\pi^{}_I(x)}-\delta_{\pi^{}_I(y)}
\end{equation}
and $\pi^{}_I(p^{}_G(x,y))=\ : \Big(\btimes_{i\in G\cap I}\{x^{}_i\}\Big) \times \Big(\btimes_{i\in \bar G\cap I}\{y^{}_i\}\Big) \ : $ in line with \eqref{pG}.

Consider now any nonzero jump. If it comes from a recombination event, it must be of the form \eqref{induziert}. That means there are types $x^{}_I, y^{}_I\in X_I$ and a subset $H$ of $I$ such that  $(\pi^{}_I.z)(x^{}_I)$ and $(\pi^{}_I.z)(y^{}_I)$ both decrease by one and the frequencies of the marginal types arising in the recombination event corresponding to $H$  increase. The rate for this transition is then given by the sum of all transitions of the original process that induce this transition in the marginal process: 
\begin{equation}\label{marg_raten}
\begin{split}
\sum\limits_{\substack{G\subset S:\\ G\cap I=H}}\sum\limits_{\substack{x\in X:\\ \pi^{}_I(x)=x^{}_I}}\sum\limits_{\substack{y\in X:\\ \pi^{}_I(y)=y^{}_I}}\frac{\varrho^{}_G}{N}z(x)z(y)
&=\sum_{\substack{G\subset S:\\ G\cap I=H}}\frac{\varrho^{}_G}{N}\big(\pi^{}_I.z\big)\big(x^{}_I\big)\cdot\big(\pi^{}_I.z\big)\big(y^{}_I\big)\\
&=\frac{\varrho^{(I)}_{H}}{N}\big(\pi^{}_I.z\big)(x^{}_I)\cdot\big(\pi^{}_I.z\big)(y^{}_I),
\end{split} 
\end{equation}
with $\varrho^{(I)}_{H}:=\sum\limits_{G\subset S: G\cap I=H}\varrho^{}_G$. So, this last term  depends only on the current state of the marginal process $(\pi^{}_I.Z_t)_{t\ge 0}$.

A mutation event of an individual of type $x$ at site $i$ from allele $x^{}_i$ to allele $y^{}_i$ induces the following transition of $\pi^{}_I.Z_t$:
\[ \pi^{}_I.z\rightarrow \pi^{}_I.z + \pi^{}_I.\delta^{}_{(\dots,y^{}_i,\dots)}-\pi^{}_I.\delta^{}_{(\dots,x^{}_i,\dots)}.\]
This jump is zero if $i\notin I$. Obviously, the transition rate is $\mu^i_{x^{}_iy^{}_i}(\pi^{}_I.z)(\pi^{}_I(x))$ and depends merely on the current state of $\pi^{}_I.Z_t$, too. The case of resampling is treated analogously. 
\end{proof}

This proof is an example for the so-called lumping procedure for Markov chains, compare \cite{burke, kemeney} for the general context or \cite{baakeetal} for the sequence context considered here.

\begin{remark}\label{subsyst}
A comparison between \eqref{recombination} and \eqref{marg_raten} shows that the marginal process $(\pi^{}_I.Z_t)_{t\ge 0}$ can itself be considered as a recombination process on the sites $I$. So, assertions about $Z_t$ will also hold for all derived marginal processes.
\end{remark}

\section{Recombination alone}
In this Section we  restrict ourselves to the case \emph{without mutation and resampling}, that means with $\mu^i_{x_iy_i}=b=0$ for all $i\in S$ and $x^{}_i,y^{}_i\in X_i$.

Since  $\v=0$ for some $x,y \in X$, there are `empty' recombination events at positive rate, but including these redundancies makes the rates in \eqref{recombination} so simple. The rates become considerably more complicated if only  `true jumps' are considered. This is already visible in the projection onto a single type. Let $x\in X$ and $Z_t=z$. In order to figure out the rate for the transition $z(x)\rightarrow z(x)+1$, we first determine the set of all pairs of types $\tilde x,\tilde y\in X$ such that for a given $G\subset S$ the jump $v^{}_{G,\tilde x,\tilde y}(x)$ equals $1$:
\begin{equation*}\begin{split}
\{\{\tilde x,\tilde y\}\subset X:v^{}_{G,\tilde x,\tilde y}(x)=1\}
&=\{\{\tilde x,\tilde y\}\subset X: \pi^{}_{G}(\tilde x)=\pi^{}_G(x), \pi^{}_{\bar G}(\tilde y)=\pi^{}_{\bar G}(x), \tilde x\neq x,\tilde y\neq x\}\\
&=\{\{\tilde x,\tilde y\}\subset X: \tilde x\in \pi^{-1}_{G}\big(\pi^{}_G(x)\big)\setminus\{x\},\tilde y\in \pi^{-1}_{\bar G}\big(\pi^{}_{\bar G}(x)\big)\setminus\{x\}\}.
\end{split}\end{equation*}
This leads to the transition rate
\begin{equation}\label{transrate1}
\sum_{G\subset S} \frac{\varrho^{}_G}{2N}\big[\big(\pi^{}_G.z\big)\big(\pi^{}_G(x)\big)-z(x)\big]\big[\big(\pi^{}_{\bar G}.z\big)\big(\pi^{}_{\bar G}(x)\big)-z(x)\big].
\end{equation}
The transition rate for $z(x)\rightarrow z(x)-1$ can be figured out analogously; $v^{}_{G,\tilde x,\tilde y}(x)=-1$ iff $\tilde x=x$ and $\tilde y$ is any type which is neither in $\pi^{-1}_G\big(\pi^{}_G(x)\big)$ nor in $\pi^{-1}_{\bar G}\big(\pi^{}_{\bar G}(x)\big)$. Since $\pi^{-1}_G\big(\pi^{}_G(x)\big)\cap\pi^{-1}_{\bar G}\big(\pi^{}_{\bar G}(x)\big)=\{x\}$ one has the rate
\begin{equation}\label{transrate2}
\sum_{G\subset S}\frac{\varrho^{}_G}{2N}z(x)\big[N-\big(\pi^{}_G.z\big)\big(\pi^{}_G(x)\big)-\big(\pi^{}_{\bar G}.z\big)\big(\pi^{}_{\bar G}(x)\big)+z(x)\big].
\end{equation}
Our aim is now to reformulate the process with the help of  additional random variables, so that the transition rates become simpler, in particular, unaffected by empty events. To this end, we define two new counting measures derived from $\Zt$, namely $(U_t)^{}_{t\ge 0}$ by $U_0(x)=0$ and 
\begin{equation*}
 U_t(x)=\text{number of events at which}\ x \text{-individuals are created until time}\ t
\end{equation*}
and $(V_t)^{}_{t\ge 0}$ by  $V_0(x)=0$ and
\begin{equation*}
 V_t(x)=\text{number of events at which}\ x \text{-individuals are broken up until time}\ t.
\end{equation*}
These processes also count events at which $Z_t$ does not change, namely the case that individuals of type $x$ are created and broken up at the same time. This may happen when an individual of type $x$ recombines according to $G$ with an individual of type $y$ with $\pi^{}_G(y)=\pi^{}_G(x)$.  Whenever this occurs, both counters increase but their difference remains unchanged. Altogether, we thus have
\begin{equation}\label{ZinUV}
Z_t=Z_0+U_t-V_t
\end{equation}
with the transition rates of $U_t$ and $V_t$ unaffected by `empty' events: For $U_t(x)=u$,  $u\rightarrow u+1$ happens at rate
\[ \sum_{G\subset S} \frac{\varrho^{}_G}{2N}\big(\pi^{}_G.z\big)\big(\pi^{}_G(x)\big)\cdot\big(\pi^{}_{\bar G}.z\big)\big(\pi^{}_{\bar G}(x)\big),\]
and for $V_t(x)=v$, the transition $v\rightarrow v+1$ happens at rate
\[\sum_{G\subset S}\frac{\varrho^{}_G}{2}z(x).\]

In the following, marginal processes will emerge frequently. We introduce a short-hand, symbolic notation similar to the one described in \cite{reco}. Fix an arbitrary $x\in X$ and define for a subset $G=\{g^{}_1,\dots,g^{}_{|G|}\}$ of sites
\begin{equation*}
[G]_t:=[g^{}_1,\dots,g^{}_{|G|}]_t:=\big(\pi^{}_G.Z_t\big)\big(\pi^{}_G(x)\big).
\end{equation*}
$[G]_t$ is the number of individuals that are identical to $x$ at the sites corresponding to $G$, at time $t$. Again, we use shorthands $[g^{}_1,\dots,g^{}_{|G|}]_t$ instead of $[\{g^{}_1,\dots,g^{}_{|G|}\}]_t$. Note that we suppress the dependence on $x$ in $[G]_t$ for ease of notation. Analogously, we define for the processes $U_t$ and $V_t$:
\begin{equation*}
 \langle G \rangle_t:=\big(\pi^{}_G.U_t\big)\big(\pi^{}_G(x)\big),
\end{equation*}

\begin{equation*}
(G)_t:=\big(\pi^{}_G.V_t\big)\big(\pi^{}_G(x)\big).
\end{equation*}

By Remark \ref{subsyst}, we can now consider $[G]_t$ as a recombination process on $|G|$ sites evaluated at the type $(x^{}_{g^{}_1},\dots,x^{}_{g^{}_{|G|}})$. 

For $|G|=2$, the distribution of $\langle g^{}_1,g^{}_2\rangle_t$ can be given explicitly because the transition rates
\begin{equation}\label{delta}
 \sum_{H\subset S: |H\cap G|=1}\frac{\varrho^{}_H}{2N}[g^{}_1]_t[g^{}_2]_t=:a  
\end{equation}
are constant in time because all 1-site marginals are constant in time. So, $\langle g^{}_1,g^{}_2\rangle_t$ follows a Poisson distribution with parameter $at$.

\subsection{Analysis of the expectation}
Since we will use it frequently, we want to recall an elementary fact concerning the dynamics of the mean of a continuous-time Markov chain with a finite state space, which is often used implicitly. The proof is a straightforward exercise that can be found  in \cite[Fact 1]{baakeherms}, for example. 

\begin{lemma}\label{erwabl2}
Let $(Z_t)_{t\ge 0}$ be a Markov process with finite state space $E \subset \mathbb{Z}^d$ with transition rates $q(z,z+v)$ for transitions from $z$ to $z+v$ for $z\in E$, $v \neq 0$ (let $q(z, z+v)=0$ if $z+v \notin E$). Then the following equation holds for all $t\ge 0$
\[ \frac {d}{dt} \EE(Z_t)=\EE(F(Z_t)), \]
where $F$ is defined as
\[ F(z):= \sum_{v\in \mathbb{Z}^d} vq(z,z+v). \]\qed
\end{lemma}

Lemma \ref{erwabl2} together with the representation of $Z_t$ in \eqref{ZinUV} gives us the dynamics of the mean:
\begin{equation}\label{erwabl}
\frac{d}{dt}\EE\Big[[1,\dots,n]_t\Big]=\sum\limits_{G \subset S }\EE\Big[\frac{\varrho^{}_G}{2N}([G]_t[\bar G]_t-N\cdot[1,\dots,n]_t)\Big].
\end{equation}

The motivation for this comes from the well-understood special case of single crossovers \cite{baakeherms}. Here, all recombination rates that are attached to multiple crossover recombination events vanish. This affects all $\varrho^{}_G$ with $G$ that either do not contain  $1$ or $n$, or have gaps.

In this case, the induced marginal processes are conditionally independent of each other and so moment closure is immediate \cite[Lemma 1 and Theorem 1]{baakeherms}:

\[\begin{split}
\frac{d}{dt}&\EE\Big[[1,\dots,n]_t\Big]=\sum\limits_{G \subset S }\EE\Big[\frac{\varrho^{}_G}{2N}([G]_t[\bar G]_t-N\cdot [1,\dots,n]_t)\Big]\\
=& \sum\limits_{G \subset S}\frac{\varrho^{}_G}{2N}(\EE\Big[[G]_t\Big]\EE\Big[[\bar G]_t\Big]-N\cdot\EE\Big[[1,\dots,n]_t\Big]).
\end{split}
\]
We obtain a finite nonlinear system of differential equations, whose solution is known in closed form \cite{reco}.

The independence relies on two properties. First, a single crossover recombination event induces a pair of marginal processes $[G]_t, [\bar G]_t$ for which $\{G,\bar G\}$ is an \textit{ordered} partition of $S$. Second, a single crossover recombination event only affects one of the induced processes while leaving the other one constant.

With general recombination both these properties are violated. First, marginal processes arise that are given by \textit{non-ordered} partitions, so even single-crossover recombination events may affect both processes at the same instant. Second, a multiple-crossover recombination event may affect the frequency of a pair of marginals that are given by an ordered partition. So, the independence of the induced marginal processes is violated in two ways.

Let us now look at \eqref{erwabl} again. On the right-hand side, an expectation of products emerges. This is what one may expect due to the inherent nonlinearity of the recombination process. Nevertheless, we see that no site arises more than once, so the arising products are described by a partition of sites. This leads us to the following question: Given an arbitrary partition of sites, what is the dynamics of the mean of the product of the induced marginal processes? Theorem \ref{zentral} below answers this. For its formulation we need the following definition.

\begin{defin}
Let $\{A_j\}_{j\in J}$ be a collection of sets with $A_i\cap A_j=\varnothing, \ i\neq j$. Define $A_J:=\bigcup\limits_{j\in J}A_j$. 
Then, $G\subset A_J$ \textit{disrupts} $\{A_j\}_{j\in J}$, denoted by $G|\{A_j\}_{j\in J}$, if
$G\cap A_j\notin \{\varnothing, G\}$ for all $j\in J$.  For $|J|=1$, we simply write $G|A_j$.
\end{defin}
Note that for a collection of pairwise disjoint subsets of sites $\{A_j\}_{j\in J}$ disrupted by $G$, in a recombination event corresponding to $G$ between individuals of marginal types  $\pi^{}_G(x)$ and $\pi^{}_{A_J\setminus G}(x)$, the processes $\langle A_j\rangle_t,\ j\in J$  increase. 
Similarly, in the recombination event corresponding to $G$  between individuals of marginal types $\pi^{}_{A_K}(x)$ and $\pi^{}_{A_J\setminus A_K}(x)$ for $K\subset J$, the processes $(A_j)_t, \ j\in J$ increase.

With these preparations, we are now ready to state
\begin{theorem}\label{zentral}

Let $m\le n$, $M:=\{1,\dots,m\}$ and $\mathcal A:=\{A_1,\dots,A_m\}$ be a partition of $\{1,\dots,n\}$. Define $\mathcal P$ as the set of all triples $(I,J,K)$, where $\{I,J,K\}$ is a partition of $M$. Then
 \begin{equation}\label{zentral_gl}
\begin{split}
\frac{d}{dt}\EE\Big[\prod\limits_{\ell\in M}[A_\ell]_t\Big]
=\EE\Big[\sum\limits_{\substack{(I,J,K)\in \mathcal P\\ I\neq M}}\prod\limits_{i\in I}[A_i]_t
\sum\limits_{\tilde K\subset K}\sum\limits_{\substack{G\subset  A_J\\ G|\{A_j\}_{j\in J}}}\frac{\varrho^{I}_{K, G}}{4N}(-1)^{|K|}[A_{\tilde K} \cup G]_t[A_{K\setminus \tilde K}\cup G^c]_t \Big],
\end{split}
\end{equation}
where $G^c$ is  the complement of $G$ in $A_J$ and $\varrho^{I}_{K,G}$ is defined as 
\begin{equation}\label{rhos}
\varrho^{I}_{K,G}:= \sum\limits_{D \subset A_I}\sum\limits_{\substack{H\subset A_K\\ H|\{A_k\}_{k\in K}}}\varrho^{}_{H\cup D \cup G}.
\end{equation}

\end{theorem}

\begin{remark}\label{Bedeutung_IJK}
The right-hand side of \eqref{zentral_gl} may be read in the following way. The set $I$ indicates the parts of $\mathcal A$ that remain unchanged under the corresponding recombination event, the sets $J$ and $K$ indicate sets for which the  derived processes $U_t$ and $V_t$, respectively, increase. So the splitting of $Z_t$ into $U_t$ and $V_t$ does not only simplify the calculation but also shows up in the result.
\end{remark}

\begin{proof}[Proof of Theorem \ref{zentral}]
For $\delta t>0$, define 
\[\langle A_\ell\rangle^t_{\delta t}:=\langle A_\ell\rangle_{t+\delta t}-\langle A_\ell\rangle_t\]
and
\[(A_\ell)^t_{\delta t}:=(A_\ell)_{t+\delta t}-(A_\ell)_t\]
then
\[[A_\ell]_{t+\delta t}=[A_\ell]_t+\langle A_\ell\rangle^t_{\delta t}-(A_\ell)^t_{\delta t}\]

 and $\prod_{\ell\in M}[A_\ell]_{t+\delta t}$ reads
\[\prod_{\ell\in M}[A_\ell]_{t+\delta t}=\sum_{(I,J,K)\in\mathcal P}(-1)^{|K|}\prod_{i\in I}[A_i]_t\prod_{j\in J}\langle A_j\rangle^t_{\delta t}\prod_{k\in K}(A_k)^t_{\delta t}.  \]
Let $t + \delta t$ be the time of the first recombination event after time $t$.\\
Then, a summand $\prod_{i\in I}[A_i]_t\prod_{j\in J}\langle A_j\rangle^t_{\delta t}\prod_{k\in K}(A_k)^t_{\delta t}$ may evaluate to:
\begin{itemize}
\item zero if there is any $j\in J$ or $k\in K$ such that $\langle A_j\rangle^t_{\delta t}=0$ or $(A_k)^t_{\delta t}=0$ 
\item $(-1)^{|K|}\prod_{i\in I}[A_i]_t$ otherwise, that means if $\langle A_j\rangle^t_{\delta t}=(A_k)^t_{\delta t}=1$ for all $j\in J$, $k\in K$.
\end{itemize}
The latter transition comes from recombination events that correspond to the union of some $G$ disrupting $\{A_j\}_{j\in J}$ and $H$ disrupting $\{A_k\}_{k\in K}$ and any subset $D$ of $A_I$. 
At such recombination events, the recombining individuals must be of the following form: $x$-alleles at $G$, $G^c$ resp., $x$-alleles at $A_k, \ k\in K$, whereas the particular $A_k$ may be arbitrarily distributed  across the two individuals (but the individual sets may not be disrupted!). Thus, the complete rate reads
\begin{equation}
r(I, J, K):=\sum\limits_{\tilde K\subset K}\sum\limits_{\substack{G\subset  A_J\\ G|\{A_j\}_{j\in J}}}\frac{\varrho^{I}_{K, G}}{4N}[A_{\tilde K} \cup G]_t[A_{K\setminus \tilde K}\cup G^c]_t, 
 \end{equation}
with $G^c$ and $\varrho^{I}_{K, G}$ as defined above. This is the rate of the event that the terms corresponding to $J$ and $K$ increase, that means it is the rate of all recombination events such that a binding arises in each $A_j,\ j\in J$, and a binding breaks in each $A_k,\ k\in K$. 

Thus,
\[\frac{d}{dt}\EE\Big[\prod\limits_{\ell\in M}[A_\ell]_t\Big]=\EE\Big[\sum\limits_{\substack{(I,J,K)\in \mathcal P\\ I\neq M}}(-1)^{|K|}\prod\limits_{i\in I}[A_i]_t r(I,J,K)\Big],\]
which is the assertion of the theorem. 

\end{proof}

Let us now consider the implication of the theorem for the moment-closure problem. The theorem tells us that
the dynamics of the mean of a product of marginal processes defined by a partition of sites  can be described by the mean of another product of marginal processes defined by a (finer) partition of sites. Since the number of sites is finite and so is the number of partitions of sites  the moment closure approach (for the mean) directly leads to a finite and linear system of ODE's. We have thus proved

\begin{corollary}
For the Moran model with recombination alone, the moment approach closes. \qed
\end{corollary}

The size of these systems explodes with the number of sites. Nevertheless, there is much redundancy in the concrete calculation of particular means. For example, in the analysis of $\EE\big[[1,2,3,4]_t\big]$, marginal processes on three sites emerge. According to Remark \ref{subsyst}, these can be treated as recombination processes on three sites, so by a proper summation of the recombination rates, one can easily determine their solutions given the solution of the three-sites recombination process.

\subsection{Comparison with the deterministic dynamics}
We now want to compare the result of Theorem \ref{zentral} to the corresponding deterministic dynamics. To this end, let $\cM(X)$ be the space of all measures on $X$. For $G\subset S$ define the recombinator $R_G$
\footnote{This is a generalisation of the recombinator in \cite{reco}. Note that the notational similarity is deceptive because $G$ denotes sites here rather than `links' (the bonds between sites) as in \cite{reco}.}
 by
\[
 R_G(\omega):=\frac{1}{|\omega|}(\pi^{}_G.\omega)\otimes(\pi^{}_{\bar G}.\omega)
\]
with $R_G(0)=0$. Consider the following dynamical system on $\cM(X)$:
\begin{equation}\label{detsyst}
 \dot{\omega}=\sum\limits_{G\subset S}\frac{\varrho^{}_G}{2}\big(R_G-\Id\big)\omega.
\end{equation}
This is the infinite population limit of the recombination process (without and  with resampling) in the following sense. If we consider $\hat Z^N_t:=\frac{1}{N}Z_t$ and let $\lim_{N\rightarrow\infty}Z^N_0=p^{}_0$, then 
\begin{equation}\label{grossezahl}
\lim_{N\rightarrow\infty}\sup_{s\le t}|\hat Z^N_s - p^{}_s|=0
\end{equation}
with probability $1$, where $p^{}_s$ is the solution of the initial value problem \eqref{detsyst} with $\omega^{}_0=p^{}_0$. This is shown in \cite{baakeherms} for the special case of single crossovers, but it is  obvious that the proof, which is based on the general law of large numbers by Ethier and Kurtz (\cite[Thm. 11.2.1]{ethier}, see also \cite{kurtz}), may be generalised to the case of multiple crossovers.

We are now interested in the relationship between \eqref{zentral_gl} and the deterministic dynamics. If $\omega_t$ follows \eqref{detsyst}, then a  (tensor) product of marginal measures $(\pi^{}_{A_1}.\omega_t)\otimes\dots\otimes(\pi^{}_{A_m}.\omega_t)$ given by a partition of sites as in Theorem \ref{zentral} exhibits the following dynamics:
\begin{equation}\label{detprod}
\begin{split}
\frac{d}{dt}\big((&\pi^{}_{A_1}.\omega_t)\otimes\dots\otimes (\pi^{}_{A_m}.\omega_t)\big)=(\pi^{}_{A_1}.(\sum\limits_{G\subset S} \frac{\varrho^{}_{G}}{2}(R_{G}-\Id))(\omega_t))\otimes (\pi^{}_{A_2}.\omega_t) \otimes \dots\otimes(\pi^{}_{A_m}.\omega_t)\\
&+(\pi^{}_{A_1}.\omega_t)\otimes(\pi^{}_{A_2}.(\sum\limits_{G\subset S}\frac{\varrho^{}_{G}}{2}(R_{G}-\Id))\omega_t)
\otimes(\pi^{}_{A_3}.\omega_t)\otimes\dots\otimes(\pi^{}_{A_m}.\omega_t)\\
&+\dots+(\pi^{}_{A_1}.\omega_t)\otimes\dots\otimes(\pi^{}_{A_{m-1}}.\omega_t)\otimes(\pi^{}_{A_m}.(\sum\limits_{G\subset S}\frac{\varrho^{}_{G}}{2}(R_{G}-\Id))\omega_t)\\
 &=\sum\limits^m_{j=1}\sum\limits_{B\subset A_j}\varrho_B(\pi^{}_{A_1}.\omega_t)\otimes
 \dots\otimes
 \big[\frac{1}{|\omega_t|}(\pi^{}_B.\omega_t)\otimes(\pi^{}_{A_j\setminus B}.\omega_t) - (\pi^{}_{A_j}.\omega_t) \big]
 \otimes\dots\otimes(\pi^{}_{A_m}.\omega_t),
\end{split}
\end{equation}
with $\varrho^{}_B:=\sum\limits_{\substack{H\subset S\\ H\cap A_j=B}}\varrho^{}_H$.

Compare this to \eqref{zentral_gl}, and only consider summands where $|J|=1$ and $|K|=0$ or $|J|=0$ and $|K|=1$. According to Remark \ref{Bedeutung_IJK}, we can understand the corresponding transitions as `uncorrelated' events at which only one marginal process changes at a given instant. We get the following terms on the right-hand side of \eqref{zentral_gl}: 
\begin{itemize}
 \item $J=\{j\}, K=\varnothing$: 
\begin{equation}\label{stern1} 
\EE\Big[\prod\limits_{i:i\neq j}[A_i]_t\sum\limits_{\substack{G\subset  A_j\\ G| A_j}} \frac{\varrho^{M\setminus\{j\}}_{\varnothing, G}}{4N}[G]_t[A_j\setminus G]_t \Big],
\end{equation}
\item $J=\varnothing, K=\{j\}$:
\begin{equation}\label{stern2}\EE\Big[\prod\limits_{i:i\neq j}[A_i]_t\frac{\varrho^{M\setminus\{j\}}_{\{j\}, \varnothing}}{4N}(-1)[A_j]_tN \Big], \end{equation}
\end{itemize}
with (cf. \eqref{rhos})
\[\varrho^{M\setminus\{j\}}_{\{j\}, \varnothing}=\sum\limits_{D \subset A_I}\sum\limits_{\substack{H\subset A_j\\ H| A_j}}\varrho^{}_{H\cup D}=\sum\limits_{\substack{H\subset A_j\\ H| A_j}}\sum\limits_{D \subset A_I}\varrho^{}_{H\cup D} =\sum\limits_{\substack{H\subset A_j\\ H| A_j}}
\varrho^{M\setminus\{j\}}_{\varnothing, H}.\]
Adding all terms of kind \eqref{stern1} and \eqref{stern2}, one obtains the analogue of the right-hand side of \eqref{detprod}.
According to Remark \ref{Bedeutung_IJK}, the summands with $|J\cup K|\ge 2$ correspond to \textquoteleft correlated\textquoteright  \ events at which two or more marginal processes change simultaneously.

We may thus conclude that the uncorrelated events correspond to the deterministic equation. We will now show  that the correlated events are of lower order and thus tend to zero in the limit $N\rightarrow \infty$.
To this end, look at \eqref{erwabl}. The right-hand side consists of the terms $\frac{1}{N}[G]_t[\bar G]_t$ and $[1,\dots,n]_t$. They are both of order $N$, since each individual term $[\dots]_t$ is of order $N$ (which follows, for example, from \eqref{grossezahl}). Let us look at the derivative of the mean of $\frac{1}{N}[G]_t[\bar G]_t$ (cf. \eqref{zentral_gl}). The terms with $|I|=1$ (those belonging to \textquoteleft uncorrelated\textquoteright\ events) will be of order  $N$ again, whereas the terms with $I=\varnothing$ are of order  $1$. By differentiating terms such as $\EE\Big[\frac{1}{N^2}[A_1]_t[A_2]_t[A_3]_t\Big]$ and beyond, the same observation applies: the order  of summands belonging to \textquoteleft correlated\textquoteright \ events is less or equal $1$, so for the relative frequencies ($[1,\dots,n]_t / N$) the dynamics of the mean tends to the dynamics of the deterministic model.

\subsection{Two sites, arbitrary moments}
In the case of two sites, the recombination process is rather simple. This mainly relies on the fact that the transition rate for $\langle 1,2\rangle_t$ is constant, as we have already seen in \eqref{delta}. Furthermore the set of partitions of two sites is trivial. In this special case we can easily show moment closure for arbitrary moments. The simplicity of the setting permits to look at $[1,2]_t$ itself without considering $\langle 1,2\rangle_t$ and $(1,2)_t$. The process $[1,2]^m_t$ has the following possible transitions (cf. \eqref{transrate1}, \eqref{transrate2}):
\[[1,2]^m_t\rightarrow ([1,2]_t +1)^m \quad\text{at rate}\quad \frac{\varrho^{}_1}{N}([1]_t-[1,2]_t)([2]_t-[1,2]_t)\]
and
\[[1,2]^m_t\rightarrow ([1,2]_t - 1)^m \quad\text{at rate}\quad \frac{\varrho^{}_1}{N}[1,2]_t(N-[1]_t-[2]_t +[1,2]_t).\]

Using the binomial theorem and eliminating empty transitions, we obtain for the $m$-th moment:
\begin{equation*}\begin{split}
\frac{d}{dt}\EE\Bigl[[1,2]^m_t\Bigr]=&\frac{\varrho^{}_1}{N�}\sum\limits^{m-1}_{k=0}\EE\Bigl[\binom{m}{k}[1,2]^k_t([1]_t-[1,2]_t)([2]_t-[1,2]_t)\Bigr]\\
&+\frac{\varrho^{}_1}{N�}\sum\limits^{m-1}_{k=0}(-1)^{m-k}\EE\Bigl[\binom{m}{k}[1,2]^k_t[1,2]_t(N-[1]_t-[2]_t+[1,2]_t)\Bigr]\\
=\sum\limits^{m-2}_{k=0}\EE\Bigl[&\frac{\varrho^{}_1}{N�}\binom{m}{k}[1,2]^k_t\{[1]_t[2]_t-2\delta^{(2)}_{m-k}[1,2]_tc+2\delta^{(2)}_{m-k+1}[1,2]^2_t+(-1)^{m-k}[1,2]_tN\}\Bigr]\\
&+\frac{\varrho^{}_1m}{N�}\EE\Bigl[[1,2]^{m-1}([1]_t[2]_t-[1,2]_tN)\Bigr],
\end{split} 
\end{equation*}
with $c:=([1]_t+[2]_t)$ and
\[\delta^{(2)}_{m-k}:=\begin{cases}
 1 &\text{if}\quad m-k\equiv 0 \ \text{mod} \ 2\\
 0 &\text{otherwise}.
         \end{cases}\]
So all emerging terms are moments of order $m$ or less.

\section{Recombination and Mutation}

We now want to add mutation to our process. Let us first look at the process with mutation alone, e.g. $b=\varrho^{}_G=0$. By Lemma \ref{erwabl2}, the derivative of the mean is:
\[\frac{d}{dt}\EE\Big[[1,\dots,n]_t\Big]=\sum\limits_{j\in S}\Big(\sum\limits_{y\in X_j}\mu^j_{yx^{}_j}[1,\dots,j-1,j+1,\dots,n]_t-\sum\limits_{y\in X_j\setminus\{x^{}_j\}}\mu^j_{x^{}_jy}[1,\dots,n]_t\Big).\]
So, it only consists of linear terms and marginal processes. When we consider a product of marginal processes given by a partition of sites as in the previous section, we have, due to the fact that mutation only acts on single sites independently of others:

\[\begin{split}
\frac{d}{dt}\EE\Big[\prod_{\ell\in M}[A_\ell]_t\Big]
&=\EE\Big[\sum\limits_{\ell\in M}\Big(\prod\limits_{i\in M\setminus \{\ell\}}[A_i]_t\sum\limits_{j\in A_\ell}\sum\limits_{y\in X_j}\mu^{j}_{yx^{}_j}[A_\ell\setminus \{j\}]_t \Big)\\
&-\sum\limits_{\ell\in M}\Big(\prod\limits_{i\in M\setminus \{\ell\}}[A_i]_t\sum\limits_{j\in A_\ell}\sum\limits_{y\in X_j\setminus\{x^{}_j\}}\mu^{j}_{x^{}_jy}[A_\ell]_t \Big) \Big].
\end{split}\]
What happens when we add recombination? Let $F_M$ and $F_R$, respectively, be the \textquoteleft mean rate of change functions\textquoteright\ from Lemma \ref{erwabl2} for $\prod_{\ell\in M}[A_\ell]_t$ from the process with solely mutation and recombination, respectively. Since mutation and recombination proceed independently, the respective function $F_{RM}$ for the recombination-mutation process is then just  $F_R+F_M$ and according to Lemma \ref{erwabl2} we have:
\[\begin{split}
\frac{d}{dt}\EE\Big[\prod_{\ell\in M}[A_\ell]_t\Big]&=\EE\Big[\sum\limits_{\substack{(I,J,K)\in \mathcal P \\ I\neq M}} \prod\limits_{i\in I}[A_i]_t(-1)^{|K|}\sum\limits_{\tilde K\subset K}
\sum\limits_{\substack{G\subset  A_J\\ G|\{A_j\}_{j\in J}}}\frac{\varrho^{I}_{K, G}}{4N}[A_{\tilde K} \cup G]_t[A_{K\setminus \tilde K}\cup G^c]_t \\
&+\sum\limits_{\ell\in M}\Big(\prod\limits_{i\in M\setminus \{\ell\}}[A_i]_t\sum\limits_{j\in A_\ell}\sum\limits_{y\in X_j}\mu^{j}_{yx^{}_j}[A_\ell\setminus \{j\}]_t \Big)\\
&-\sum\limits_{\ell\in M}\Big(\prod\limits_{i\in M\setminus \{\ell\}}[A_i]_t\sum\limits_{j\in A_\ell}\sum\limits_{y\in X_j\setminus\{x^{}_j\}}\mu^{j}_{x^{}_jy}[A_\ell]_t \Big) \Big].
\end{split}\]
So, the arising terms are the same as in the pure recombination process plus linear terms. It is therefore clear that we have moment closure here as well.

\section{Recombination and Resampling}
In this section, we set $b>0$ and the mutation rates zero again, so we look at the Moran model  with recombination and resampling only. At first glance, one may think that resampling has no effect on the expectation, since the process with resampling alone has a constant mean.  Indeed, the first derivative of the mean looks the same as in the pure recombination case:
\begin{equation}\label{resampling_2sites_1}
\frac{d}{dt�}\EE\big[[1,2]_t\big]=\frac{\varrho^{}_1}{N�}\EE\big[[1]_t[2]_t-[1,2]_tN\big].
\end{equation}
However, due to resampling, the one-site marginal processes are no longer constant, so we do not have instantaneous moment closure any more (cf. \eqref{resampling}: at a resampling event, the frequency of alleles may change). The  derivative of their product is obtained after an elementary but lengthy calculation:
\begin{equation}\label{resampling_2sites_2}
\frac{d}{dt}\EE\big[[1]_t[2]_t\big]= \frac{b}{�N}\EE\big[[1,2]_tN-[1]_t[2]_t\big].  
\end{equation}
We obtain a finite linear system of differential equations, namely \eqref{resampling_2sites_1} and \eqref{resampling_2sites_2}. In particular, we obtain
\begin{equation}\label{lde}
\frac{d}{dt}\EE\big[[1,2]_tN-[1]_t[2]_t\big]= -\big(\frac{\varrho^{}_1}{N}+\frac{b}{�N}\big)\EE\big[[1,2]_tN-[1]_t[2]_t\big]
\end{equation}
with the obvious exponential solution. The term $[1,2]_tN-[1]_t[2]_t$ is a correlation function, a so called linkage disequilibrium, which is widely used in population genetics. We see that both, recombination and resampling, reduce correlations between sites. 

For more than two sites, exact moment closure can no longer be established. To make this plausible, we will only present the derivative of $\EE\big[[1]_t[2]_t[3]_t\big]$ (again, the calculation is elementary but lengthy), which is a term that emerges in the derivatives of the process on three sites due to recombination: 
\[\begin{split}
\frac{d}{dt�}\EE&\big[[1]_t[2]_t[3]_t\big]=\frac{b}{N�}\EE\big[[1]_t[2,3]_tN+[2]_t[1,3]_tN+[3]_t[1,2]_tN+[1]_t[1,2]_t[1,3]_t\\
&+[2]_t[1,2]_t[2,3]_t
+[3]_t[1,3]_t[2,3]_t
-3[1]_t[2]_t[3]_t-[1]^2_t[1,2,3]_t-[2]^2_t[1,2,3]_t-[3]^2_t[1,2,3]_t\big].
\end{split}\]
The last three terms are quadratic and it is clear that further differentiating will lead to terms such as $[1]^3_t[2]_t[3]_t$ whose derivative will contain moments of even higher order. Thus, the interaction between recombination and resampling destroys moment closure.

\section{Conclusion}

In this paper, we have extended the single-crossover Moran model from \cite{baakeherms} to include general  recombination.
The dynamics of the expectation under general recombination becomes significantly more complicated. In particular, it now deviates from the dynamics in the infinite population model. The reason is the loss of independence of certain marginal processes.  

As is usual with nonlinear processes, the dynamics of a given moment requires higher moments.  Nevertheless, in this case after a finite number of steps no additional terms emerge. This is due to the fact that the arising processes may in each step  be described by a partition of sites. When mutation is included, this exact moment closure persists, but the arising processes can no longer be described by a partition of sites. Altogether, we have an exception to the rule that the dynamics of the moments of nonlinear processes lead to infinite hierarchies of ODE's.

This exact moment closure gets lost when we extend the model to include genetic drift (i.e., resampling). This is, of course, disappointing since the Moran model with recombination alone is mathematically interesting, but of limited biological value. Nevertheless, the resulting hierarchy of moments might be interesting to analyse with respect to the various possibilities of approximate moment closure.

Furthermore, the arising terms such as $\EE\big[\prod\limits_{\ell\in M}[A_\ell]_t\big]$ are of considerable interest in population genetics beyond this moment closure procedure, since they are the building blocks of the linkage disequilibria \cite{burger} that are so important in population genetics (compare \eqref{lde} for the simplest example).

\end{document}